\documentclass[12 pt]{article}
\usepackage{authblk, amsmath, amsthm, amssymb, setspace, mathtools, graphics, graphicx, caption,epstopdf}
\usepackage{subfigure, subfig, float,booktabs,dcolumn,pifont}

\title{Pairwise approximation for \textit{SIR} type network epidemics with non-Markovian recovery}

\author{
	G. R\"ost$^{1}$, Z. Vizi$^{1}$ and I.Z. Kiss$^{2}$}

\affil{$^{1}$Bolyai Institute, University of Szeged, 
	Aradi v\'ertan\'uk tere 1, 
	Szeged 6720, Hungary\\
	$^{2}$School of Mathematical and Physical Sciences, 
	Department of Mathematics,
	University of Sussex,
	Falmer, Brighton BN1 9QH, UK\\
}

\begin{document}
\maketitle
\newtheorem{lemma}{Lemma}
\newtheorem{corollary}{Corollary}
\newtheorem{theorem}{Theorem}
\newcommand{\defeq}{\vcentcolon=}
\newcommand{\eqdef}{=\vcentcolon}

\begin{abstract}
We present the generalised mean-field and pairwise models for non-Markovian epidemics on networks with arbitrary recovery time distributions. First we consider a hyperbolic system, where the population of infective nodes and links are structured by age since infection. By solving the partial differential equations, the model is transformed into a system of integro-differential equations, which is analysed both from a mathematical and numerical point of view. We analytically study the asymptotic behaviour of the generalised model and provide a rigorous mathematical proof of the conjecture on the functional form of the final epidemic size and pairwise reproduction number. As an illustration of the applicability of the general model we recover known results for the exponentially distributed and fixed recovery time cases and obtain new pairwise models with gamma and uniformly distributed infectious period. The general framework that we proposed shows a more complete picture of the impact of non-Markovian recovery on network epidemics.
\end{abstract}

\section{Introduction}
It has long been acknowledged that the connectivity pattern between individuals in a population is an important factor in determining the properties of a disease spreading process 
\cite{newman2002spread,newman2003structure,keeling2005networks,boccaletti2006complex}. Using networks to model disease transmission, where individuals are represented as nodes in a network and the connectivity between individuals is represented by links between the nodes, allowed us to capture a high level of detail of many realistic processes and led to more accurate models, especially when compared to classical compartmental models which usually operate on the assumption of homogeneous random mixing.  The most popular node-level model is perhaps the degree-based or heterogeneous mean-field models, with pairwise models offering an explicit treatment of the epidemic process both at node and link level \cite{keeling1999effects,kissrostvizi}. Such and other models of epidemic dynamics of networks (see \cite{pastor2014epidemic} for a review) has led to a much better understanding of the role of contact heterogeneity, assortativity and clustering of contacts, to name just a few important network characteristics or measures.

While networks offer an accurate representation of contact patterns, many network epidemic models only consider Markovian epidemics with both the infectious and recovery process being memoryless. However,  empirical observations show that assuming Markovian infectiousness is a strong simplifying and unrealistic assumption \cite{distrib,lloyd2001realistic}. Unfortunately, the analysis of non-Markovian systems is significantly more challenging and requires deeper theoretical and numerical tools. This paper is motivated by the renewed interest in non-Markovian processes 
\cite{boguna2014simulating, cooper2013non,jo2014analytically,min2013suppression} and aims to extend the pairwise model from Markovian to non-Markovian epidemic dynamics where the infection process remains Markovian but the infectious period is taken from an arbitrary distribution. This paper is structured as follows: in Section \ref{modelderivation}, we present the derivation of pairwise and mean-field models for arbitrary infectious periods, with the most technical part of the calculations presented in the Appendix. Section \ref{generalresults} includes the mathematical analysis of the resulting systems, with focus on the positivity of solutions, associated reproduction numbers and the implicit relation concerning the final epidemic size. In Section \ref{specialcases} we illustrate, that in special cases the resulting system can be reduced to the well-known and previously studied models. In the Appendix we collect the most technical parts of proofs, namely the calculations involved in the model derivation and the algebraic manipulations of the resulting multiple integrals.

\section{Model derivation} \label{modelderivation}
We consider an undirected and unweighted network with $N$ nodes and an average degree $n$. Each node has a state, which can be susceptible ($S$), infected ($I$) or recovered ($R$) in our context. If an infection happens along an $S-I$ link, then the state of susceptible node $S$ changes to $I$. Each infected node has a random infectious period and after it elapses, the node recovers and changes its state to $R$ permanently, which means that the infectious period is the same as the recovery time in our context and recovered nodes remain immune. 

We want to build mean-field and pairwise models for the $SIR$ type epidemic process with Markovian (i.e. exponentially distributed) transmission and general recovery time distribution. We use the notations $[X](t)$, $[XY](t)$ and $[XYZ](t)$ to express the expected number of nodes in state $X$, links in state $X-Y$ and triplets in state $X-Y-Z$, respectively, where $X,Y,Z\in\{S,I,R\}$. For the derivation of a self-consistent model at the level of nodes (i.e. for mean-field model), we obtain equations for $\dot{[S]}(t)$ and $\dot{[I]}(t)$, and these depend on the expected number of pairs. For the pairwise model (at the level of links/pairs) equations for $\dot{[SS]}(t)$ and $\dot{[SI]}(t)$ are needed, which in turn depend on triples.

First, let $i(t,a)$ represent the density of infected nodes with respect to the age of infection $a$ at the current time $t$, then $[I](t)=\int_{0}^{\infty}i(t,a)da$. Similarly, $Si(t,a)$ and $ISi(t,a)$ describe the density of $S-i$ links and $I-S-i$ triplets, respectively, where the infected node $i$ has age $a$ at time $t$ and $[SI](t)=\int_{0}^{\infty}Si(t,a)da$, $[ISI](t)=\int_{0}^{\infty}ISi(t,a)da$. We assume that the infection process along $S-I$ links is Markovian with transmission rate $\tau>0$. The recovery part is considered to be non-Markovian, with a cumulative distribution function $F(a)$ and probability density function $f(a)$. We use the associated survival function $\xi(a)=1-F(a)$ and hazard function $h(a)=-\frac{\xi'(a)}{\xi(a)}=\frac{f(a)}{\xi(a)}$. 

Using the notations above, we arrive at the following model
\begin{subequations}
	\begin{align}
	\dot{[S]}(t)&=-\tau [SI](t),\label{eq:eqS}\\
	\left(\frac{\partial}{\partial t}+\frac{\partial}{\partial a}\right)i(t,a)&=-h(a) i(t,a),\label{eq:eqi}\\
	\dot{[SS]}(t)&=-2 \tau [SSI](t)\label{eq:eqSS},\\
	\left(\frac{\partial}{\partial t}+\frac{\partial}{\partial a}\right)Si(t,a)&=-\tau ISi(t,a)-(\tau+h(a))Si(t,a),\label{eq:eqSi}
	\end{align}
\end{subequations}
subject to the boundary conditions
\begin{subequations}
	\begin{align}
	i(t,0)&=\tau [SI](t),\label{eq:initi}\\
	Si(t,0)&=\tau [SSI](t),\label{eq:initSi}
	\end{align}
\end{subequations}
and initial conditions
\begin{subequations}
	\begin{align}
	[S](0)&=[S]_0,[SS](0)=[SS]_0, i(0,a)=\varphi (a),\label{eq:initi2}\\
	Si(0,a)&=\chi(a)\approx\frac{n}{N} [S]_0 i(0,a)=\frac{n}{N} [S]_0 \varphi (a).\label{eq:initSi2}
	\end{align}
\end{subequations}
We shall use the biologically feasible assumption $\lim_{a\rightarrow\infty}\varphi(a)=0$. To break the dependence on higher order moments, we apply the closure approximation formula 
\begin{equation}
\label{eq:closure}
[XSY](t)=\frac{n-1}{n}\frac{[XS](t)[SY](t)}{[S](t)},
\end{equation}
which can be also applied for $ISi(t,a)$ in the form 
\begin{equation}
\label{eq:closure2}
ISi(t,a)=\frac{n-1}{n}\frac{[SI](t)Si(t,a)}{[S](t)}.
\end{equation}
To obtain a self-consistent system for classical network variables $[S]$, $[SS]$, $[I]$ and $[SI]$, further calculations are needed, which are presented in Appendix \ref{appendixmodelderiv}. The resulting pairwise system is the following integro-differential equation:
\begin{subequations}
	\label{eq:closeq}
	\begin{align}
	\dot{[S]}(t)&=-\tau [SI](t) \label{eq:closedeqS}\\
	\dot{[SS]}(t)&=-2\tau \frac{n-1}{n} \frac{[SS](t) [SI](t)}{[S](t)} \label{eq:closedeqSS}\\
	\dot{[I]}(t)&=\tau [SI](t) - \int_{0}^{t} \tau [SI](t-a) f(a) da - \int_{t}^{\infty} \varphi(a-t) \frac{f(a)}{\xi(a-t)} da\label{eq:closedeqI}\\
	\dot{[SI]}(t)&=\tau \frac{n-1}{n}\frac{[SS](t)[SI](t)}{[S](t)}-\tau \frac{n-1}{n}\frac{[SI](t)}{[S](t)}[SI](t)-\tau [SI](t)\nonumber\\
	&-\int_0^t \tau \frac{n-1}{n}\frac{[SS](t-a)[SI](t-a)}{[S](t-a)} e^{-\int_{t-a}^t \tau\frac{n-1}{n}\frac{[SI](s)}{[S](s)}+\tau ds} f(a)da\nonumber \\
	&-\int_t^{\infty}\frac{n}{N} [S]_0 \varphi (a-t) e^{-\int_{0}^t 
		\tau\frac{n-1}{n}\frac{[SI](s)}{[S](s)} +\tau ds} \frac{f(a)}{\xi(a-t)}da. \label{eq:closedeqSI}
	\end{align}
\end{subequations}
From Eq.(\ref{eq:closeq}), the associated mean-field model can be easily deduced by using the closure approximation formula for homogeneous networks (i.e. $n$-regular graphs)
\begin{equation}
\label{eq:closeformmf}
[XY](t)=\frac{n}{N}[X](t)[Y](t),
\end{equation}
thus the node-level system becomes
\begin{subequations}
	\label{eq:closmfeq}
	\begin{align}
	\dot{S}(t)&=-\tau \frac{n}{N}S(t)I(t) \label{eq:closedmfeqS}\\
	\dot{I}(t)&=\tau \frac{n}{N}S(t)I(t) - \int_{0}^{t} \tau \frac{n}{N}S(t-a)I(t-a) f(a) da- \int_{t}^{\infty} \varphi(a-t) \frac{f(a)}{\xi(a-t)} da.\label{eq:closedmfeqI}
	\end{align}
\end{subequations}
In the following, we investigate these systems from mathematical and numerical point of view, focussing on the epidemiologically meaningful properties of the models.

\section{General results}\label{generalresults}

In this section, we explore the most important features of systems (\ref{eq:closeq}) and (\ref{eq:closmfeq}). First, we find a first integral of the pairwise model (\ref{eq:closeq}), which allows us to reduce the dimensionality. We show that the solutions of the models are biologically meaningful, i.e. solutions with nonnegative data remain nonnegative for $t\geq0$. The paramount results of this part are the implicit relations between the reproduction number and the final epidemic size. We summarize the definitions of the associated reproduction numbers referring to \cite{kissrostvizi}, where the basic ($\mathcal{R}_0$) and pairwise ($\mathcal{R}_0^p$) reproduction numbers are precisely introduced for mean-field and pairwise models, respectively.
\subsection{First integral}\label{generalresultsfirstintegral} 
We use (\ref{eq:closedeqS}) and (\ref{eq:closedeqSS}) to find an invariant quantity of the system.
\begin{lemma}
	The function $U(t)=\frac{[SS](t)}{[S]^{2 \frac{n-1}{n}}(t)}$ is a first integral of system (\ref{eq:closeq}).
\end{lemma}
\begin{proof}
	To see this, let us divide Eq.(\ref{eq:closedeqSS}) by Eq.(\ref{eq:closedeqS}), which gives
	\begin{eqnarray*}
		\frac{d[SS]}{d[S]}=\frac{-2\tau \frac{n-1}{n}\frac{[SS] [SI]}{[S]}}{-\tau [SI]}=2 \frac{n-1}{n} \frac{[SS]}{[S]}.
	\end{eqnarray*}
	Solving this equation, we find $[SS]=K [S]^{2\frac{n-1}{n}}$, where $K$ is a constant, thus $U(t)=\frac{[SS](t)}{[S]^{2\frac{n-1}{n}}(t)}$ is an invariant quantity in the system and its value is 
	\begin{equation*}
	U(0)=K=\frac{[SS](0)}{[S]^{2\frac{n-1}{n}}(0)}=\frac{[SS]_0}{[S]_0^{2\frac{n-1}{n}}}=\frac{n [S]_0 \frac{[S]_0}{N}}{[S]_0^{2\frac{n-1}{n}}}=\frac{n}{N} [S]^{\frac{2}{n}}_0.
	\end{equation*}	
\end{proof}
\noindent Consequently, using this first integral, we obtain
\begin{equation}
\label{eq:firstintegral}
[SS](t)=\frac{n}{N} [S]^{\frac{2}{n}}_0 [S]^{2\frac{n-1}{n}}(t). 
\end{equation}
Applying Eq.(\ref{eq:firstintegral}), we can reduce our pairwise model to a two-dimensional system:
\begin{eqnarray}
\label{eq:hom2dimsys0}
\dot{[S]}(t)&=&-\tau [SI](t),\nonumber \\
\dot{[SI]}(t)&=&\tau \kappa [S]^{\frac{n-2}{n}}(t)[SI](t) -\tau [SI](t)-\tau\frac{n-1}{n} \frac{[SI](t)}{[S](t)}[SI](t)\nonumber \\ 
&& -\int_{0}^{t}\tau \kappa [S]^{\frac{n-2}{n}}(t-a)[SI](t-a) e^{-\int_{t-a}^t \tau\frac{n-1}{n}\frac{[SI](s)}{[S](s)}+\tau ds} f(a) da\nonumber \\
&& -\int_t^{\infty}\frac{n}{N} [S]_0 \varphi (a-t) e^{-\int_{0}^t \tau\frac{n-1}{n}\frac{[SI](s)}{[S](s)}+\tau ds} \frac{f(a)}{\xi(a-t)} da,
\end{eqnarray}
where 
\begin{equation*}
\kappa=\frac{n-1}{N} [S]^{\frac{2}{n}}_0.
\end{equation*}
\subsection{Positivity}\label{generalresultspositivity}
We are interested only in nonnegative solutions of system (\ref{eq:closeq}). The following proposition shows, that the solutions remain nonnegative provided that the initial conditions are nonnegative.
\begin{lemma}
	If initial conditions $[S]_0$, $[SS]_0$ are nonnegative and $\varphi(a)\geq0$ for $a\geq0$, then $[S](t)\geq0$, $[SS](t)\geq0$, $[I](t)\geq0$ and $[SI](t)\geq0$ hold for $t\geq0$.   
\end{lemma}
\begin{proof}
	It is clear, that $[SS](t)$ remains nonnegative, if the initial condition $[SS](0)$ is nonnegative, because $[SS](t)$ can be expressed from Eq.(\ref{eq:closedeqSS}) in the form
	\begin{equation*}
	[SS](t)=[SS]_0\,e^{-2\tau\frac{n-1}{n}\int_{0}^{t}\frac{[SI](s)}{[S](s)} ds}.
	\end{equation*}
	Moreover, if $[SS]_0$ is positive, then $[SS](t)>0$ for all $t\geq0$. From Eq. (\ref{eq:firstintegral}) we obtain that $[S](t)$ cannot be zero, if $[SS](t)$ is positive for all $t\geq0$, which implies (from continuity of solutions) $[S](t)>0$ for $t\geq0$. From the derivation of system (\ref{eq:closeq}) (see Appendix), we have the following formulae for $[I](t)$ and $[SI](t)$:
	\begin{equation}
	\label{eq:closedformI}
	[I](t)=\int_{0}^{t} \tau [SI](t-a) \xi(a) da + \int_{t}^{\infty} \varphi(a-t) \frac{\xi(a)}{\xi(a-t)} da,
	\end{equation}
	and
	\begin{eqnarray}
	\label{eq:closedformSI}
	[SI](t)&=&\int_0^t \tau \frac{n-1}{n}\frac{[SS](t-a)[SI](t-a)}{[S](t-a)} e^{-\int_{t-a}^t \tau\frac{n-1}{n}\frac{[SI](s)}{[S](s)}+\tau ds} \xi(a) da\nonumber\\
	&&+
	\int_t^{\infty}\frac{n}{N} [S]_0 \varphi (a-t) e^{-\int_{0}^t \tau\frac{n-1}{n}\frac{[SI](s)}{[S](s)}+\tau ds} \frac{\xi(a)}{\xi(a-t)} da. 
	\end{eqnarray}
	It can be seen that $[I](t)$ remains nonnegative if $[SI](t)$ is nonnegative for $t\geq0$. On the other hand, $[SI](t_0)$ cannot be zero for some $t_0\geq0$, because the right-hand side of (\ref{eq:closedformSI}) depends on the $[S](t;t<t_0)$, $[SS](t;t<t_0)$ and $[S](t;t<t_0)$, which are positive, hence $[SI](t)>0$.
\end{proof}
\noindent In the case of the mean-field model (\ref{eq:closmfeq}), the positivity of $S(t)$ is clear. To see the positivity of $I(t)$, we substitute (\ref{eq:closeformmf}) into (\ref{eq:closedformI}), which gives
\begin{equation}
\label{eq:closedmfformI}
I(t)=\int_{0}^{t} \tau\frac{n}{N} S(t-a)I(t-a) \xi(a) da + \int_{t}^{\infty} \varphi(a-t) \frac{\xi(a)}{\xi(a-t)} da.
\end{equation}
Notice that $I(t)$ remains nonnegative if $S(t)$ is nonnegative for $t\geq0$.

\subsection{Reproduction numbers $\mathcal{R}_0$ and $\mathcal{R}_0^p$}\label{generalresultsreproductionnumbers}

To determine the reproduction numbers for our models, we start with the usual interpretation, which specifies $\mathcal{R}_0$ as the number of secondary infections generated by a 'typical' infected individual introduced into a fully susceptible population during its infectious period. In \cite{kissrostvizi} the reproduction numbers are precisely described in both cases: in context of mean-field model, we use the \textit{basic} reproduction number $\mathcal{R}_0$, which is the expected lifetime of an $I$ node multiplied by the number of newly generated $I$ nodes per unit time. On the other hand, the \textit{pairwise} reproduction number $\mathcal{R}_0^p$ is the expected lifetime of an $S-I$ link multiplied by the number of newly generated $S-I$ links per unit time. These definitions above give
\begin{equation}
\label{eq:R0init}
\mathcal{R}_0=\tau [SI]_0 \mathbb{E}(\mathcal{I}),
\end{equation}
where $\mathbb{E}(\mathcal{I})$ denotes the expected value of the random variable $\mathcal{I}$ defined by the infectious period of an infected node, and
\begin{equation}
\label{eq:R0pinit}
\mathcal{R}^p_0=\tau [SSI]_0 \mathbb{E}(Z)=\tau [SSI]_0 \frac{\left(1-\mathcal L[f_\mathcal{I}](\tau)\right)}{\tau},
\end{equation}
where $\mathbb{E}(Z)$ denotes the expected value of the random variable $Z$ defined by the lifetime of an $S-I$ link and $\mathcal L[f_{\mathcal{I}}](\tau)$ denotes the Laplace transform of $f_{\mathcal{I}}$, the density of the recovery process at $\tau$. Applying the mean-field closure assumption (\ref{eq:closeformmf}) for (\ref{eq:R0pinit}), we get 
\begin{equation}
\label{eq:R0}
\mathcal{R}_0=\tau \frac{n}{N} S_0 \mathbb{E}(\mathcal{I}),
\end{equation}
and using the pairwise closure approximation (\ref{eq:closure}) and first integral (\ref{eq:firstintegral}) in (\ref{eq:R0pinit}), we find 
\begin{equation}
\label{eq:R0p}
\mathcal{R}^p_0=\frac{n-1}{N} [S]_0\left(1-\mathcal L[f_\mathcal{I}](\tau)\right).
\end{equation}
We omit the detailed calculations here (see \cite{kissrostvizi}). 
\subsection{Final size relation}\label{generalresultsfinalsize}
In this part, we derive final size relations that allow us to calculate the total number of infected nodes during an epidemic outbreak on the network. We use the notation $s_\infty=\frac{[S]_\infty}{[S]_0}$, where $[S]_\infty = \lim_{t\rightarrow\infty}[S](t)$ and $1-s_{\infty}$ is called the attack rate (the fraction of infected nodes).
\begin{theorem}
	The final size relation associated to the mean-field model (\ref{eq:closmfeq}) is 
	\begin{eqnarray}
	\label{eq:finalsizemftheorem}
	\ln\left(s_\infty\right)=\mathcal{R}_0\left(s_\infty-1\right), 
	\end{eqnarray}
	where the basic reproduction number $\mathcal{R}_0$ is defined in (\ref{eq:R0}).
\end{theorem}

\begin{proof}
	From (\ref{eq:closedmfeqS}), we obtain
	\begin{equation}
	\label{eq:sinfs0}
	S_\infty-S_0=-\tau \frac{n}{N}\int_0^{\infty} S(u) I(u) du
	\end{equation}
	and
	\begin{equation}
	\label{eq:finbegin}
	\ln\left(\frac{S_\infty}{S_0}\right)=-\tau \frac{n}{N}\int_0^{\infty}I(u) du.
	\end{equation}
	Substituting (\ref{eq:closedmfformI}) into (\ref{eq:finbegin}), we get
	\begin{eqnarray*}
		\ln\left(\frac{S_\infty}{S_0}\right)&=&-\tau \frac{n}{N}\int_0^{\infty}\int_{0}^{u} \tau\frac{n}{N} S(u-a)I(u-a) \xi(a) da du\nonumber \\
		&&- \tau \frac{n}{N}\int_0^{\infty}\int_{u}^{\infty} \varphi(a-u) \frac{\xi(a)}{\xi(a-u)} da du.
	\end{eqnarray*}
	Neglecting the small number of initial infecteds ($\varphi(a)\approx 0$), we obtain
	\begin{eqnarray}
	\label{eq:beforemffs}
	\ln\left(\frac{S_\infty}{S_0}\right)&=&-\left(\tau \frac{n}{N}\right)^2\int_0^{\infty}\int_{0}^{u} S(u-a)I(u-a) \xi(a) da du.
	\end{eqnarray}
	After some algebraic manipulation (for details, see Appendix \ref{appendixfinalsizemftheorem}), we obtain
	\begin{eqnarray}
	\label{eq:aftermffs}	
	\ln\left(\frac{S(\infty)}{S(0)}\right)=\tau \frac{n}{N}\mathbb{E}(\mathcal{I})\left(S(\infty)-S(0)\right), 
	\end{eqnarray}
	where $\mathcal{I}$ denotes the infectious period of an infected node. Therefore, we found
	\begin{equation*}
	\ln\left(s_\infty\right)=\mathcal{R}_0\left(s_\infty-1\right),
	\end{equation*}
	where $\mathcal{R}_0=\tau \frac{n}{N}\mathbb{E}(\mathcal{I})S_0$.
\end{proof}
In the following, we derive the final-size relation for the pairwise system (\ref{eq:closeq}).
\begin{theorem}
	The final size relation associated to the pairwise model (\ref{eq:closeq}) is 
	\begin{eqnarray}
	\label{eq:finalsizetheorem}
	\frac{s_\infty^{\frac{1}{n}}-1}{\frac{1}{n-1}}=\mathcal{R}_0^p
	\left(s_\infty^{\frac{n-1}{n}}-1\right), 
	\end{eqnarray}
	where the pairwise reproduction number $\mathcal{R}_0^p$ is defined in (\ref{eq:R0p}).
\end{theorem}   

\begin{proof}
	The second equation of the two-dimensional system (\ref{eq:hom2dimsys0}) has the general form
	\begin{equation*}
	x'(t)=\alpha(t)-\beta(t) x(t),
	\end{equation*} 
	where 
	\begin{eqnarray*}
		\alpha(c)&=&\tau \kappa [S]^{\frac{n-2}{n}}(c)[SI](c) \nonumber \\
		&&-\int_{0}^{c}\tau \kappa [S]^{\frac{n-2}{n}}(c-a)[SI](c-a) f(a) e^{-\int_{c-a}^{c} \tau\frac{n-1}{n}\frac{[SI](s)}{[S](s)}+\tau ds} da\nonumber \\
		&& -\int_{c}^{\infty} \frac{n}{N}[S]_0 \varphi (a-c) \frac{f(a)}{\xi(a-c)} e^{-\int_{0}^c \tau\frac{n-1}{n}\frac{[SI](s)}{[S](s)}+\tau ds}da,\nonumber \\
		\beta(w)&=&\tau+\tau\frac{n-1}{n}\frac{[SI](w)}{[S](w)},\nonumber\\
		x(t)&=&[SI](t),
	\end{eqnarray*}
	and has the solution 
	\begin{equation}
	\label{eq:genSIsol}
	x(u)= e^{-\int_0^u \beta(w)dw} x(0) + \int_0^u  e^{-\int_c^u \beta(w)dw}\alpha(c) dc.
	\end{equation}
	Using (\ref{eq:eqS}), simple calculations give the relations
	\begin{eqnarray*}
		e^{-\int_{0}^{u}\tau+\tau\frac{n-1}{n}\frac{[SI](s)}{[S](s)} ds}&=& e^{-\tau u}\frac{[S]^{\frac{n-1}{n}}(u)}{[S]^{\frac{n-1}{n}}(0)},\nonumber \\
		e^{-\int_{c}^{u}\tau+\tau\frac{n-1}{n}\frac{[SI](s)}{[S](s)} ds}&=& e^{-\tau (u-c)}\frac{[S]^{\frac{n-1}{n}}(u)}{[S]^{\frac{n-1}{n}}(c)}.
	\end{eqnarray*}
	Using these relations, from (\ref{eq:genSIsol}) we get 
	\begin{eqnarray*}
		[SI](u)&=& \frac{[SI](0)}{[S]^{\frac{n-1}{n}}(0)}e^{-\tau u} [S]^{\frac{n-1}{n}}(u) + \int_{0}^{u}\tau\kappa[S]^{-\frac{1}{n}}(c) [SI](c) e^{\tau c} e^{-\tau u} [S]^{\frac{n-1}{n}}(u) dc \nonumber \\ 
		&& - \int_{0}^{u} \int_{0}^{c} \tau \kappa [S]^{-\frac{1}{n}}(c-a) [SI](c-a) f(a) e^{-\tau a} e^{\tau c} e^{-\tau u} [S]^{\frac{n-1}{n}}(u)\;da\;dc \nonumber \\
		&& - \int_{0}^{u} \int_{c}^{\infty} \frac{n}{N} [S]^{\frac{1}{n}}_0 \varphi(a-c) \frac{f(a)}{\xi(a-c)} e^{-\tau u} [S]^{\frac{n-1}{n}}(u)\;da\;dc. 	
	\end{eqnarray*}
	Then, substituting this formula into the first equation of (\ref{eq:hom2dimsys0}), we find an equation in general form 
	\begin{eqnarray}
	\label{eq:1dimsys}
	[S]'(t)=-\tau A(t) [S]^{\frac{n-1}{n}}(t),	
	\end{eqnarray}
	where 
	\begin{eqnarray}
	\label{eq:homAeq}
	A(u)&=&\frac{[SI](0)}{[S]^{\frac{n-1}{n}}(0)}e^{-\tau u}+\int_{0}^{u}\tau\kappa[S]^{-\frac{1}{n}}(c) [SI](c) e^{\tau c} e^{-\tau u}\;dc \nonumber \\ 
	&& - \int_{0}^{u} \int_{0}^{c} \tau \kappa [S]^{-\frac{1}{n}}(c-a) [SI](c-a) f(a) e^{-\tau a} e^{\tau c} e^{-\tau u}\;da\;dc \nonumber \\
	&& - \int_{0}^{u} \int_{c}^{\infty} \frac{n}{N} [S]^{\frac{1}{n}}_0 \varphi(a-c) \frac{f(a)}{\xi(a-c)} e^{-\tau u}\;da\;dc.
	\end{eqnarray}
	Solving this scalar equation, we have
	\begin{eqnarray*}
		[S]^{\frac{1}{n}}(t)=[S]^{\frac{1}{n}}_0-\frac{\tau}{n} \int_{0}^{t} A(u) du.	
	\end{eqnarray*}
	For the final size relation, we consider the following equation
	\begin{eqnarray}
	\label{eq:homfinallast}
	[S]^{\frac{1}{n}}_{\infty}=[S]^{\frac{1}{n}}_0-\frac{\tau}{n} \int_{0}^{\infty} A(u) du.	
	\end{eqnarray}
	Using the linearity of integration, we have to calculate the following four integrals on the right-hand side:
	\begin{eqnarray*}
		I_1&=&\int_{0}^{\infty} \frac{[SI](0)}{[S]^{\frac{n-1}{n}}(0)}e^{-\tau u} \;du,\\
		I_2&=&\int_{0}^{\infty} \int_{0}^{u}\tau\kappa[S]^{-\frac{1}{n}}(c) [SI](c) e^{\tau c} e^{-\tau u} \;dc \;du,\\
		I_3&=&\int_{0}^{\infty} \int_{0}^{u} \int_{0}^{c} \tau \kappa [S]^{-\frac{1}{n}}(c-a) [SI](c-a) f(a) e^{-\tau a} e^{\tau c} e^{-\tau u}da dcdu,\\
		I_4&=&\int_{0}^{\infty} \int_{0}^{u} \int_{c}^{\infty} \frac{n}{N} [S]^{\frac{1}{n}}_0 \varphi(a-c) \frac{f(a)}{\xi(a-c)} e^{-\tau u} \;da \;dc\;du.
	\end{eqnarray*}
	After lengthy calculations (see Appendix \ref{appendixfinalsizepw}), we arrive to the relation
	\begin{eqnarray}
	\label{eq:homfinalsize}
	[S]^{\frac{1}{n}}_{\infty}=[S]^{\frac{1}{n}}_0+
	\kappa \frac{1}{n-1}\left(1-\int_{0}^{\infty} f(a) e^{-\tau a} da\right)\left([S]^{\frac{n-1}{n}}_\infty-[S]^{\frac{n-1}{n}}_0\right).
	\end{eqnarray} 
	After some algebraic manipulation and substituting back the formula of $\kappa$, we have
	\begin{eqnarray}
	\label{eq:homfslap}
	\frac{s_\infty^{\frac{1}{n}}-1}{\frac{1}{n-1}}=\frac{n-1}{N}\left(1-\mathcal{L}[f](\tau)\right)[S]_0
	\left(s_\infty^{\frac{n-1}{n}}-1\right), 
	\end{eqnarray}
	where $\mathcal L[f](\tau)$ denotes the Laplace transform of $f$, the PDF of recovery time at $\tau$. 
\end{proof}
\noindent Notice that for $n\rightarrow\infty$, the relation (\ref{eq:finalsizetheorem}) takes the form $\ln(s_\infty)=\mathcal{R}_0^p(s_\infty-1)$, which is exactly the classical form of final size relations.
\section{Special cases}\label{specialcases}
In this section, we investigate some common choices for the recovery time. As we expect, if $\mathcal{I}\sim \mathrm{Exp}(\gamma)$ (i.e. the infectious period $\mathcal{I}$ is exponentially distributed), we get back the classical Markovian models. In the case of fixed recovery time, the models reduce to the systems studied in details in \cite{kissrostvizi}. We can also recover the multi-stage infection model of \cite{sherborneblyusskiss} with gamma distributed recovery time. Finally, we consider $\mathcal{I}\sim\mathrm{Uniform}(A,B)$ and write down the associated equations in a compact form. In this section we assume, that the initial infecteds are 'newborn', i.e. the initial distribution of infected nodes $\varphi(a)=[I]_0\delta(a)$, where $\delta(a)$ is the Dirac delta function. Then, 
\begin{equation}
\label{eq:newbornI}
\int_{t}^{\infty} \varphi(a-t) \frac{f(a)}{\xi(a-t)} da=[I]_0 f(t)
\end{equation}
and 
\begin{equation}
\label{eq:newbornSI}
\int_t^{\infty}\frac{n}{N} [S]_0 \varphi (a-t) e^{-\int_{0}^t \tau\frac{n-1}{n}\frac{[SI](s)}{[S](s)} +\tau ds} \frac{f(a)}{\xi(a-t)}da=\frac{n}{N} [S]_0 e^{-\int_{0}^t 
	\tau\frac{n-1}{n}\frac{[SI](s)}{[S](s)} +\tau ds} f(t).
\end{equation}
\subsection{Exponential distribution with parameter $\gamma$}\label{specialcasesexp}
The most widely used distribution in disease modelling is the exponential distribution. Both the stochastic and deterministic approaches exploit the memorylessness property to build tractable models. The resulting deterministic systems are ordinary differential equations, which are very favoured in epidemiology because of their simple structure and numerical solvability.
In the exponential case, $\xi(t)=e^{-\gamma t}$ and $f(t)=\gamma e^{-\gamma t}$. Using the assumption $\varphi(a)=[I]_0\delta(a)$, (\ref{eq:closedformI}) and $f(t)=\gamma\,\xi(t)$, from (\ref{eq:closedeqI}) we obtain
\begin{eqnarray*}
	\dot{[I]}(t)=\tau [SI](t)-\gamma [I](t),
\end{eqnarray*}
which gives the classical Markovian type pairwise equation for $[I](t)$. With similar arguments, from (\ref{eq:closedeqSI}) we obtain
\begin{eqnarray*}
	\dot{[SI]}(t)&=&\tau \frac{n-1}{n}\frac{[SS](t)[SI](t)}{[S](t)}-\tau \frac{n-1}{n}\frac{[SI](t)}{[S](t)}[SI](t)-\tau [SI](t)\nonumber \\
	&& -\gamma [SI](t).
\end{eqnarray*}
For the mean-field model (\ref{eq:closmfeq}), the same calculation gives the classical Markovian mean-field equation for $\dot{I}(t)$:
\begin{eqnarray*}
	\dot{[I]}(t)=\tau S(t) I(t)-\gamma I(t).
\end{eqnarray*} 
\subsection{Fixed recovery time $\sigma$}\label{specialcasesfixed}
In several models, it is a reasonable assumption for the infectious period to have a fixed, constant duration, e.g. for measles \cite{baileymeasles}. In the case of fixed recovery time $\sigma$, we have
\[
\xi(t)= \begin{cases}
1 & \textrm{if $0\leq t<\sigma$,} \\
0 & \textrm{if $t\geq\sigma$}, \\	
\end{cases}
\]
and
\[
f(t)=\delta(t-\sigma), 
\]
where $\delta(t)$ denotes the Dirac-delta function. Applying the fundamental property of $\delta(t)$, from (\ref{eq:closedeqI}) and $\varphi(a)=[I]_0 \delta(a)$ we have
\[ 
\dot{[I]}(t)=\tau [SI](t)-\begin{cases}
0 & \textrm{if $0\leq t<\sigma$,} \\
\tau [SI](t-\sigma) & \textrm{if $t\geq\sigma$}, \\
\end{cases}
\]
and from (\ref{eq:closedeqSI}), if $0\leq t<\sigma$, we obtain
\begin{eqnarray*}
	\dot{[SI]}(t)&=&\tau \frac{n-1}{n}\frac{[SS](t)[SI](t)}{[S](t)}-\tau \frac{n-1}{n}\frac{[SI](t)}{[S](t)}[SI](t)-\tau [SI](t).
\end{eqnarray*}
If $t>\sigma$, we get 
\begin{eqnarray*}
	\dot{[SI]}(t)&=&\tau \frac{n-1}{n}\frac{[SS](t)[SI](t)}{[S](t)}-\tau \frac{n-1}{n}\frac{[SI](t)}{[S](t)}[SI](t)-\tau [SI](t)\nonumber \\ 
	&& -\tau \frac{n-1}{n}\frac{[SS](t-\sigma)[SI](t-\sigma)}{[S](t-\sigma)} e^{-\int_{t-\sigma}^t \tau\frac{n-1}{n}\frac{[SI](s)}{[S](s)}+\tau ds}.
\end{eqnarray*}
which is exactly the same system, that was studied in detail in \cite{kissrostvizi}. The mean-field model for fixed recovery time (see \cite{kissrostvizi}) can also be derived from (\ref{eq:closedmfeqI}) using the same arguments.
\subsection{Gamma distribution with shape $K\in\mathbb{Z}^+$ and rate $K \gamma$}\label{specialcasesgamma}
The case of gamma distributed recovery time was studied in \cite{sherborneblyusskiss}. 
Using pairwise approximation with a standard closure, the authors have been able to analytically derive a number of important characteristics of disease dynamics. These included the final size of an epidemic and the epidemic threshold. Their results have shown that a higher number of disease stages, but with the same average duration of the infectious period, results in faster epidemic outbreaks with a higher peak prevalence and a larger final size of the epidemic. The pairwise model in \cite{sherborneblyusskiss} has the following equations for nodes:
\begin{eqnarray}
\label{eq:gammapairwise}
\dot{[S]}&=& -\tau \sum_{i=1}^{K} [SI_i],\nonumber\\
\dot{[I_1]}&=& \tau \sum_{i=1}^{K} [SI_i]-K \gamma [I_1] \nonumber\\
\dot{[I_j]}&=& K \gamma [I_{j-1}] -K \gamma [I_j],\;\;j=2,3,\dots K,
\end{eqnarray}
where $I_i,\;i=1,2,\dots, K$ are the infectious stages, where nodes spend an exponentially distributed time with parameter $K \gamma$. The distribution of the total infectious period is the sum of $K$ exponential distributions with parameter $K \gamma$, which gives the gamma distribution with shape $K$ and rate $K \gamma$ (thus the expected infectious period is $K \times 1/K \gamma= 1/\gamma$). Clearly, $[I](t)=\sum_{j=1}^{K}[I_i](t)$ and $[SI](t)=\sum_{i=1}^{K}[SI_i](t)$ and the sum of equations for infectious stages gives 
$$\dot{[I]}(t)=\tau [SI](t)-K \gamma [I_K](t).$$
On the other hand, using (\ref{eq:newbornI}), the PDF and survival function of Gamma distribution 
\begin{eqnarray*}
	f(a) &=& \frac{(K \gamma)^K}{(K-1)!}a^{K-1}e^{-K\gamma a},\nonumber\\
	\xi(a) &=& e^{-K \gamma a}\sum_{k=0}^{K-1}\frac{(K\gamma)^k}{k!}a^k,
\end{eqnarray*} 
and inserting into (\ref{eq:closedeqI}) and (\ref{eq:closedformI}), we have 
\begin{eqnarray}
\label{eq:gammaeqIdot}
\dot{[I]}(t)&=&\tau [SI](t)-K\gamma \left( \int_{0}^{t}\tau [SI](t-a)\frac{(K \gamma)^{K-1}}{(K-1)!}a^{K-1}e^{-K \gamma a}da- [I]_0 \frac{(K \gamma)^{K-1}}{(K-1)!}t^{K-1}e^{-K \gamma t}\right)\nonumber\\ 
\end{eqnarray}
and
\begin{eqnarray}
\label{eq:gammaeqI}
[I](t)=\sum_{k=0}^{K-1}\left(\int_{0}^{t}\tau [SI](t-a)\frac{(K \gamma)^k}{k!}a^{k}e^{-K \gamma a}da -[I]_0 \frac{(K \gamma)^k}{k!}t^{k}e^{-K \gamma t} \right).
\end{eqnarray}
These equations suggest the relations 
\begin{eqnarray}
\label{eq:gammaeqIj}
[I_j](t)=\int_{0}^{t}\tau [SI](t-a)\frac{(K \gamma)^{j-1}}{(j-1)!}a^{j-1}e^{-K \gamma a}da + [I]_0 \frac{(K \gamma)^{j-1}}{(j-1)!}t^{j-1}e^{-K \gamma t},\;j=1,2,\dots,K.
\end{eqnarray}
To show this, we consider the equations for infectious stages in (\ref{eq:gammapairwise}) as a first-order, linear differential equations with variation of constants formulae
\begin{eqnarray}
\label{eq:variationofconstantsI1}
[I_1](t) &=& [I_1](0) e^{-K \gamma t}+\int_{0}^{t} e^{-K \gamma (t-s)}\tau [SI](s)ds
\end{eqnarray}
and 
\begin{eqnarray}
\label{eq:variationofconstantsIj}
[I_j](t) &=& [I_j](0) e^{-K \gamma t}+\int_{0}^{t} e^{-K \gamma (t-s)}K \gamma [I_{j-1}](s)ds,\;\; j=2,3,\dots,K.
\end{eqnarray} 
If all infecteds are newborn, we have $[I_1](0)=[I]_0$ and $[I_2](0)=[I_3](0)=\dots=[I_K](0)=0$. Proceeding by induction yields that (\ref{eq:gammaeqIj}) satisfies (\ref{eq:variationofconstantsI1}) for $j=1$ and (\ref{eq:variationofconstantsIj}) for $j=2,3,\dots,K$ (for details, see Appendix \ref{appendixgammadistribution}). It is analogous to derive the equations for $[SI_j](t)$.
\subsection{Uniform distribution on interval $[A,B]$}\label{specialcasesuniform}
The uniform distribution is one of the most natural probability distributions and preferred in agent-based modeling \cite{laskowskiABM}, and was applied also for avian influenza \cite{yangavian}. Let the recovery time be distributed uniformly on interval $[A,B]$ (we assume $0<A<B$), i.e. 
\[ 
f(t)=\begin{cases}
\frac{1}{B-A} & \textrm{if $t\in (A,B)$,} \\
0 & \textrm{otherwise}, \\
\end{cases}
\]
and
\[ 
\xi(t)=\begin{cases}
1 & \textrm{if $t\leq A$,} \\
\frac{B-t}{B-A} & \textrm{if $t\in (A,B)$,} \\
0 & \textrm{if $t\geq B$}. \\
\end{cases}
\]
We have to study the three cases $t<A$, $A<t<B$ and $t>B$. Writing the equation for $\dot{[I]}(t)$, we have (after changing the variable):
\[ 
\dot{[I]}(t=\tau [SI](t)-\begin{cases}
0 & \textrm{if $t< A$,} \\
\int_{0}^{t-A}\frac{\tau [SI](u)}{B-A}du+\frac{[I]_0}{B-A} & \textrm{if $t\in [A,B]$,} \\
\int_{t-B}^{t-A}\frac{\tau [SI](u)}{B-A}du & \textrm{if $t> B$}. \\
\end{cases}
\]
With a more compact notation,
\begin{equation*}
\dot{[I]}(t)=\tau [SI](t) - \int_{\max(0,t-B)}^{\max(0,t-A)} \frac{\tau [SI](u)}{B-A} du - \frac{[I]_0}{B-A}\chi_{[A,B]}(t),
\end{equation*}
where $\chi_{[A,B]}(t)$ is the indicator function of interval $[A,B]$.
The same argument gives
\begin{eqnarray*}
	\dot{[SI]}(t)&=&\tau \frac{n-1}{n}\frac{[SS](t)[SI](t)}{[S](t)}-\tau \frac{n-1}{n}\frac{[SI](t)}{[S](t)}[SI](t)-\tau [SI](t)\nonumber \\ 
	&& -\int_{\max(0,t-B)}^{\max(0,t-A)} \frac{\tau}{B-A} \frac{n-1}{n}\frac{[SS](u)[SI](u)}{[S](u)} e^{-\int_{u}^t \tau\frac{n-1}{n}\frac{[SI](s)}{[S](s)}+\tau ds} du\nonumber \\
	&& -\frac{n}{N} [S]_0 e^{-\int_{0}^t \tau\frac{n-1}{n}\frac{[SI](s)}{[S](s)}+\tau ds}\frac{[I]_0}{B-A}\chi_{[A,B]}(t).
\end{eqnarray*}
For $t>B$ the model becomes a system of differential equations with distributed delays.
\section{Discussion}\label{discussion} 
While the main focus of this paper is on the rigorous derivation and analysis of the model, we have performed a number of numerical tests where the results of explicit stochastic network simulations on networks are compared to the output from the generalised pairwise model. In Fig.~\ref{fig:fig_sim}, homogeneous (or regular random) networks were considered and the average of 100 simulations is compared to the numerical solutions of mean-field (\ref{eq:closeformmf}) and pairwise (\ref{eq:closeq}) models. Several observations can be made: (a) the agreement of the simulation results with the 
numerical solution of pairwise model is excellent, and (b) the mean-field model, which largely ignores the network structure, performs poorly. This gives us great confidence that the generalised pairwise model can and will be used in different contexts as dictated by empirical or other theoretical studies. Moreover, we can observe that for different recovery time distributions sharing the same expected value the one with the smaller variance produces the larger outbreak (see also \cite{rostvizikissbiomat}).

\begin{figure}[h!]
	\centering\includegraphics[width=5.0in]{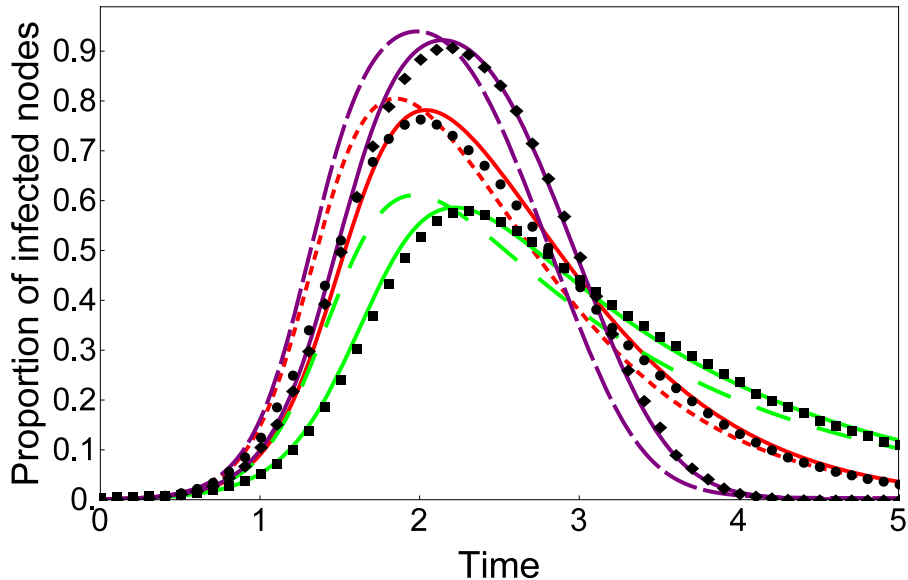}
	\caption{Stochastic and numerical experiments for non-Markovian epidemic with various recovery time distributions on homogeneous networks with $N=1000$ nodes and infection rate $\tau=0.35$. Squares, circles, diamonds show the mean of 100 simulations on random regular graphs with average degree $\langle k\rangle=15$ for exponential distribution with parameter $\lambda=\frac{2}{3}$ ($\mathrm{mean}=\frac{3}{2}$,$\mathrm{variance}=\frac{9}{4}$), gamma distribution with shape $\alpha=3$ and rate $\beta=2$ ($\mathrm{mean}=\frac{3}{2}$,$\mathrm{variance}=\frac{3}{4}$), uniform distribution on 
		interval $[a,b]=[1,2]$ ($\mathrm{mean}=\frac{3}{2}$,$\mathrm{variance}=\frac{1}{12}$), respectively. Dashed and solid lines correspond to the numerical solution of the mean-field (\ref{eq:closeformmf}) and parwise (\ref{eq:closeq}) models, respectively.} 
	\label{fig:fig_sim}
\end{figure}
The generalised pairwise model provides a description of a possible deterministic approximation of non-Markovian epidemic processes on networks. The integro-differential 
system, which describes the dynamics at the level of nodes and links, is a powerful tool for investigating the classical quantities of an $SIR$-type epidemic, such as the reproduction number and final epidemic size. Through the generalised model we have also provided a rigorous mathematical proof of the conjecture for the functional form of final size relation, which was proposed in \cite{kissrostvizi}.

The generalised model is more challenging to analyse due to its complexity but it largely relies on tools from the theory of integro-differential equations. Further extensions of the model could focus on relaxing the assumption of homogeneous networks and extend the model to networks with heterogeneous degree distribution, see for example 
\cite{NeilJTB, SK15}, or to consider modelling the situation where both the infectious and recovery processes are non-Markovian.

Several different approaches exist to model non-Markovian epidemics on networks. These are largely guided by the choice of model and variables to be tracked. Notable examples include the message passing approach, often referred to as the cavity model \cite{KN10a, WS14}, and the percolation based approach \cite{newman2002spread,KR07,Mil07, Mil08}. While the latter only offers information about the final state of the epidemic, the former
describes the temporal evolution of the epidemic. Generalisations of the pairwise model to gamma-distributed infectious periods have also been proposed and this has been developed for both homogeneous and heterogeneous networks \cite{sherborneblyusskiss,NeilJTB}. As for Markovian epidemics, we expect that many of the models mentioned above are complementary and offer different perspectives on the time evolution or final state of the epidemic. We expect that many of the non-Markovian models will in fact be equivalent under appropriately chosen initial conditions and appropriate 
averaging. With the model proposed we wanted to emphasise opportunities to frame problems and models of network epidemics in more rigorous mathematical terms and use existing mathematical theory to enhance our understanding of stochastic processes on networks and their average behaviour as captured by mean-field models.

%


\clearpage

\section{Appendix}\label{appendix}
\subsection{Calculations for model derivation} \label{appendixmodelderiv}
Repeating Eq. (\ref{eq:eqS}) and applying the moment-closure formula (\ref{eq:closure}) to Eq. (\ref{eq:eqSS}), we have
\begin{eqnarray*}
	\dot{[S]}(t)&=&-\tau [SI](t),\nonumber \\
	\dot{[SS]}(t)&=&-2\tau \frac{n-1}{n} \frac{[SS](t) [SI](t)}{[S](t)}.\nonumber
\end{eqnarray*}
Using $[I](t)=\int_{0}^{\infty}i(t,a)da$, from Eq. (\ref{eq:eqi}) we obtain
\begin{eqnarray}
\label{eq:Idot}
\dot{[I]}(t)&=&\int_{0}^{\infty}\frac{\partial}{\partial t} i(t,b) db=\int_{0}^{\infty}\left(-h(b) i(t,b)-\frac{\partial}{\partial b}i(t,b)\right)db\nonumber \\
&=& -\int_{0}^{\infty} h(b) i(t,b) db - \left(i(t,\infty)-i(t,0)\right)\nonumber \\
&=&-\int_{0}^{\infty} h(b) i(t,b) db - i(t,\infty) + i(t,0).
\end{eqnarray}
Solving the first-order linear PDE (\ref{eq:eqi}) along characteristic lines, we obtain
\begin{equation*}
\begin{split}
i(t,a)=
&\begin{cases}
i(t-a,0) e^{-\int_{0}^{a} h(b) db}, \quad \mathrm{if}\; t > a; \\
i(0,a-t) e^{-\int_{a-t}^{a} h(b) db}, \quad \mathrm{if}\; t \leq a .
\end{cases}
\end{split}
\end{equation*}
Plugging (\ref{eq:initi}) and (\ref{eq:initi2}) into the solution above, we have
\begin{equation}
\label{eq:ita}
\begin{split}
i(t,a)=
&\begin{cases}
\tau [SI](t-a) e^{-\int_{0}^{a} h(b) db}, \quad \mathrm{if}\; t > a; \\
\varphi(a-t) e^{-\int_{a-t}^{a} h(b) db}, \quad \mathrm{if}\; t \leq a .
\end{cases}
\end{split}
\end{equation}
Applying this formula for $[I](t)=\int_{0}^{\infty}i(t,a)da$, we find
\begin{eqnarray}
\label{eq:Ipde}
[I](t)=\int_{0}^{t}\tau [SI](t-a) e^{-\int_{0}^{a} h(b) db} da+\int_{t}^{\infty}\varphi(a-t) e^{-\int_{a-t}^{a} h(b) db} da.
\end{eqnarray}
Finally, using that along the characteristic lines, $i(t,\infty)=i(0,\infty)=\varphi(\infty)=0$ from the assumption, substituting (\ref{eq:ita}) and the boundary condition (\ref{eq:initi}) into (\ref{eq:Idot}), we get 
\begin{eqnarray}
\label{eq:ddeI}
\dot{[I]}(t)&=&\tau [SI](t) - \int_{0}^{t} \tau [SI](t-a) h(a) e^{-\int_{0}^{a} h(b) db} da \nonumber \\
&&- \int_{t}^{\infty} \varphi(a-t) h(a) e^{-\int_{a-t}^{a} h(b) db} da .
\end{eqnarray}
Using the definition and properties of hazard function, we can deduce the following formulae:
\begin{subequations}
	\begin{align}
	e^{-\int_{0}^{a}h(b) db}&=\frac{\xi(a)}{\xi(0)}=\xi(a),\label{eq:hazardprop1}\\
	e^{-\int_{a-t}^{a}h(b) db}&=\frac{\xi(a)}{\xi(a-t)}, \label{eq:hazardprop2}\\
	h(a) e^{-\int_{0}^{a}h(b) db}&=f(a), \label{eq:hazardtosurv1}\\
	h(a) e^{-\int_{a-t}^{a}h(b) db}&=\frac{f(a)}{\xi(a-t)}.\label{eq:hazardtosurv2}
	\end{align}
\end{subequations} 
Applying these formulae to Eq.(\ref{eq:Ipde}) and (\ref{eq:ddeI}), we have 
\begin{eqnarray}
\label{eq:Iformula}
[I](t)=\int_{0}^{t}\tau [SI](t-a) \xi(a) da+\int_{t}^{\infty}\varphi(a-t) \frac{\xi(a)}{\xi(a-t)} da,
\end{eqnarray}
and
\begin{equation*}
\dot{[I]}(t)=\tau [SI](t) - \int_{0}^{t} \tau [SI](t-a) f(a) da - \int_{t}^{\infty} \varphi(a-t) \frac{f(a)}{\xi(a-t)} da.
\end{equation*}
To compute the equation for $[SI](t)$, we follow the calculation process above. First, applying (\ref{eq:closure2}) to Eq. (\ref{eq:eqSi}), we get 
\begin{eqnarray}
\label{eq:assumpSi}
\left(\frac{\partial}{\partial t}+\frac{\partial}{\partial a}\right)Si(t,a)
&=& -\frac{\tau(n-1)}{n}\frac{[SI](t)}{[S](t)} Si(t,a) -(\tau+h(a)) Si(t,a).
\end{eqnarray} 
Using $[SI](t)=\int_{0}^{\infty}Si(t,a)da$, from Eq. (\ref{eq:assumpSi}) we find
\begin{eqnarray}
\label{eq:SIeqwithSi}
\dot{[SI]}(t)&=&\int_{0}^{\infty}\frac{\partial}{\partial t}Si(t,a)da\nonumber \\
&=&\int_{0}^{\infty}\left(-\frac{\tau(n-1)}{n}\frac{[SI](t)}{[S](t)}Si(t,a)\right)da-\int_{0}^{\infty}(\tau+h(a))Si(t,a)da -\int_{0}^{\infty}\frac{\partial}{\partial a}Si(t,a)da\nonumber \\
&=& -\tau\frac{n-1}{n}\frac{[SI](t)}{[S](t)}[SI](t) -\tau [SI](t)-\int_{0}^{\infty}h(a) Si(t,a)da -Si(t,\infty)+Si(t,0).
\end{eqnarray}  
We want to express the variable $Si(t,a)$ as a function of classical network variables. To achieve this, let us consider the following first-order PDE:
$$\left(\frac{\partial}{\partial t}+\frac{\partial}{\partial a}\right)x(t,a)=-f(t) x(t,a)-g(a)x(t,a)$$
with boundary conditions 
$$x(t,0)=\phi(t), \quad x(0,a)=\psi(a).$$
Solving along the characteristic lines $t-a=c$, we find that
\begin{equation}
\label{eq:pdesol0}
\begin{split}
x(t,a)=
&\begin{cases}
\phi(t-a) e^{-\int_{t-a}^t f(s) ds} e^{-\int_0^a g(b) db}, \quad \mathrm{if}\; t > a; \\
\psi(a-t) e^{-\int_{0}^t f(s) ds} e^{-\int_{a-t}^a g(b) db}, \quad \mathrm{if}\; t \leq a.
\end{cases}
\end{split}
\end{equation}
In our case, 
$x(t,a)=Si(t,a)$, 
$f(t)=\tau\frac{n-1}{n}\frac{[SI](t)}{[S](t)}$, 
$g(a)=\tau+h(a)$,
$\phi(t)=\tau \frac{n-1}{n}\frac{[SS](t)[SI](t)}{[S](t)}$, (from closure approximation (\ref{eq:closure2})) and $\psi(a)=\frac{n}{N} [S]_0 \varphi (a)$, hence from (\ref{eq:pdesol0}) we get
\begin{equation}
\label{eq:Si(t,a)form}
\begin{split}
Si(t,a)=
&\begin{cases}
\tau \frac{n-1}{n}\frac{[SS](t-a)[SI](t-a)}{[S](t-a)} e^{-\int_{t-a}^t \tau\frac{n-1}{n}\frac{[SI](s)}{[S](s)} ds} e^{-\int_0^a \tau+h(b) db}, \quad \mathrm{if}\; t > a; \\
\frac{n}{N} [S]_0 \varphi (a-t) e^{-\int_{0}^t \tau\frac{n-1}{n}\frac{[SI](s)}{[S](s)} ds} e^{-\int_{a-t}^a \tau+h(b) db}, \quad \mathrm{if}\; t \leq a .
\end{cases}
\end{split}
\end{equation}
Again, along the characteristic lines $Si(t,\infty)=Si(0,\infty)=\chi(\infty)=0$.
Putting (\ref{eq:Si(t,a)form}) into $[SI](t)=\int_{0}^{\infty}Si(t,a)da$, we obtain
\begin{eqnarray}
\label{eq:SIpde}
[SI](t)&=&\int_0^t \tau \frac{n-1}{n}\frac{[SS](t-a)[SI](t-a)}{[S](t-a)} e^{-\int_{t-a}^t \tau\frac{n-1}{n}\frac{[SI](s)}{[S](s)} ds} e^{-\int_0^a \tau+h(b) db} da\nonumber\\
&&+
\int_t^{\infty}\frac{n}{N} [S]_0 \varphi (a-t) e^{-\int_{0}^t \tau\frac{n-1}{n}\frac{[SI](s)}{[S](s)} ds} e^{-\int_{a-t}^a \tau+h(b) db} da.
\end{eqnarray}
If we substitute (\ref{eq:Si(t,a)form}) back to Eq.(\ref{eq:SIeqwithSi}), we derive 
\begin{eqnarray}
\label{eq:SIeqfromPDE}
\dot{[SI]}(t)&=&\tau \frac{n-1}{n}\frac{[SS](t)[SI](t)}{[S](t)} -\frac{\tau(n-1)}{n}\frac{[SI](t)}{[S](t)}[SI](t)-\tau [SI](t)\nonumber \\ 
&& -\int_0^t \tau \frac{n-1}{n}\frac{[SS](t-a)[SI](t-a)}{[S](t-a)} e^{-\int_{t-a}^t \tau\frac{n-1}{n}\frac{[SI](s)}{[S](s)} ds} e^{-\int_0^a \tau+h(b) db} h(a) da\nonumber \\
&& -\int_t^{\infty}\frac{n}{N} [S]_0 \varphi (a-t) e^{-\int_{0}^t \tau\frac{n-1}{n}\frac{[SI](s)}{[S](s)} ds} e^{-\int_{a-t}^a \tau+h(b) db} h(a) da.
\end{eqnarray}
Applying the formulae (\ref{eq:hazardprop1})-(\ref{eq:hazardtosurv2}) for Eq. (\ref{eq:SIpde}) and (\ref{eq:SIeqfromPDE}), we have
\begin{eqnarray}
\label{eq:SIformula}
[SI](t)&=&\int_0^t \tau \frac{n-1}{n}\frac{[SS](t-a)[SI](t-a)}{[S](t-a)} e^{-\int_{t-a}^t \tau\frac{n-1}{n}\frac{[SI](s)}{[S](s)}+\tau ds} \xi(a) da\nonumber\\
&&+
\int_t^{\infty}\frac{n}{N} [S]_0 \varphi (a-t) e^{-\int_{0}^t \tau\frac{n-1}{n}\frac{[SI](s)}{[S](s)}+\tau ds} \frac{\xi(a)}{\xi(a-t)} da,
\end{eqnarray}
and
\begin{eqnarray*}
	\dot{[SI]}(t)&=&\tau \frac{n-1}{n}\frac{[SS](t)[SI](t)}{[S](t)}-\tau \frac{n-1}{n}\frac{[SI](t)}{[S](t)}[SI](t)-\tau [SI](t)\nonumber \\ 
	&& -\int_0^t \tau \frac{n-1}{n}\frac{[SS](t-a)[SI](t-a)}{[S](t-a)} e^{-\int_{t-a}^t \tau\frac{n-1}{n}\frac{[SI](s)}{[S](s)}+\tau ds} f(a)da\nonumber \\
	&& -\int_t^{\infty}\frac{n}{N} [S]_0 \varphi (a-t) e^{-\int_{0}^t \tau\frac{n-1}{n}\frac{[SI](s)}{[S](s)} +\tau ds} \frac{f(a)}{\xi(a-t)}da.
\end{eqnarray*}
\subsection{Calculation for final size relation (\ref{eq:finalsizemftheorem})} \label{appendixfinalsizemftheorem}
In this section, we provide the derivation from (\ref{eq:beforemffs}) to (\ref{eq:aftermffs}):
\begin{eqnarray*}
	\ln\left(\frac{S(\infty)}{S(0)}\right)&=&-\left(\tau \frac{n}{N}\right)^2\int_0^{\infty}\int_{0}^{u} S(u-a)I(u-a) \xi(a) da du\nonumber \\
	&=& -\left(\tau \frac{n}{N}\right)^2\int_0^{\infty}\int_{0}^{u} S(v)I(v) \xi(u-v) dv du\nonumber \\
	&=&-\left(\tau \frac{n}{N}\right)^2\int_0^{\infty}\int_{v}^{\infty} S(v)I(v) \xi(u-v) du dv\nonumber \\
	&=&-\left(\tau \frac{n}{N}\right)^2\int_0^{\infty} S(v)I(v) \left[\int_{v}^{\infty}\xi(u-v) du \right]dv\nonumber \\
	&=&-\left(\tau \frac{n}{N}\right)^2\int_0^{\infty} S(v)I(v) \left[\int_{0}^{\infty}\xi(p) dp \right]dv\nonumber \\
	&=&\tau \frac{n}{N}\left[\int_{0}^{\infty}\xi(p) dp \right]\left(S(\infty)-S(0)\right)\nonumber \\
	&=&\tau \frac{n}{N}\left[\int_{0}^{\infty}\int_{p}^{\infty}f(q)dq dp \right]\left(S(\infty)-S(0)\right)\nonumber \\
	&=&\tau \frac{n}{N}\left[\int_{0}^{\infty}\int_{0}^{q}f(q)dp dq \right]\left(S(\infty)-S(0)\right)\nonumber \\
	&=&\tau \frac{n}{N}\left[\int_{0}^{\infty}q f(q) dq \right]\left(S(\infty)-S(0)\right)\nonumber \\
	&=&\tau \frac{n}{N}\mathbb{E}(\mathcal{I})\left(S(\infty)-S(0)\right).
\end{eqnarray*}
\subsection{Calculation for final size relation (\ref{eq:finalsizetheorem})} \label{appendixfinalsizepw}
Now, we compute the four integrals appear on the right-hand side of Eq. (\ref{eq:homfinallast}). For the first integral, we have
\begin{eqnarray*}
	I_1 &=& \int_{0}^{\infty} \frac{[SI](0)}{[S]^{\frac{n-1}{n}}(0)}e^{-\tau u} du= \frac{[SI](0)}{[S]^{\frac{n-1}{n}}(0)}\left[ \frac{e^{-\tau u}}{-\tau}\right]_{0}^{\infty}=  \frac{[SI](0)}{[S]^{\frac{n-1}{n}}(0)} \frac{1}{\tau}.
\end{eqnarray*}
After some algebraic manipulation, we obtain the following expression for the second integral $I_2$:
\begin{eqnarray*}
	I_2&=&\int_{0}^{\infty} \int_{0}^{u}\tau\kappa[S]^{-\frac{1}{n}}(c) [SI](c) e^{\tau c} e^{-\tau u} dc du = \int_{0}^{\infty} \int_{c}^{\infty}\tau\kappa[S]^{-\frac{1}{n}}(c) [SI](c) e^{\tau c} e^{-\tau u} du dc \nonumber \\
	&=& \int_{0}^{\infty} \tau\kappa[S]^{-\frac{1}{n}}(c) [SI](c) e^{\tau c} \left[ \frac{e^{-\tau u}}{-\tau}\right]_{c}^{\infty} dc = \int_{0}^{\infty} \tau \kappa[S]^{-\frac{1}{n}}(c) [SI](c) e^{\tau c} \frac{e^{-\tau c}}{\tau} dc\nonumber \\
	&=& -\frac{1}{\tau}\kappa \int_{0}^{\infty} [S]^{-\frac{1}{n}}(c)\dot{[S]}(c) dc=-\frac{1}{\tau}\kappa \frac{n}{n-1}\left([S]^{\frac{n-1}{n}}_\infty-[S]^{\frac{n-1}{n}}_0\right).	
\end{eqnarray*}
The most challenging one is the third integral $I_3$:
\begin{eqnarray*}
	I_3 &=&\int_{0}^{\infty} \int_{0}^{u} \int_{0}^{c} \tau \kappa [S]^{-\frac{1}{n}}(c-a) [SI](c-a) f(a) e^{-\tau a} e^{\tau c} e^{-\tau u} da dc du \nonumber \\
	&=& \int_{0}^{\infty} \int_{c}^{\infty} \int_{0}^{c} \tau \kappa [S]^{-\frac{1}{n}}(c-a) [SI](c-a) f(a) e^{-\tau a} e^{\tau c} e^{-\tau u} da du dc \nonumber \\
	&=& \int_{0}^{\infty} \int_{0}^{c} \int_{c}^{\infty}  \tau \kappa [S]^{-\frac{1}{n}}(c-a) [SI](c-a) f(a) e^{-\tau a} e^{\tau c} e^{-\tau u} du da dc \nonumber \\
	&=& \int_{0}^{\infty} \int_{0}^{c} \tau \kappa [S]^{-\frac{1}{n}}(c-a) [SI](c-a) f(a) e^{-\tau a} e^{\tau c} \left[ \frac{e^{-\tau u}}{-\tau}\right]_{c}^{\infty} da dc \nonumber \\
	&=& \frac{1}{\tau} \int_{0}^{\infty} \int_{0}^{c} \tau \kappa [S]^{-\frac{1}{n}}(c-a) [SI](c-a) f(a) e^{-\tau a} da dc \nonumber \\
	&=& -\frac{1}{\tau} \kappa \int_{0}^{\infty} f(a) e^{-\tau a}  \left[\frac{[S]^{\frac{n-1}{n}}(c-a)}{\frac{n-1}{n}}\right]_{a}^{\infty} da \nonumber \\
	&=& -\frac{1}{\tau} \kappa \frac{n}{n-1}\left([S]^{\frac{n-1}{n}}_\infty-[S]^{\frac{n-1}{n}}_0\right)\int_{0}^{\infty} f(a) e^{-\tau a} da.
\end{eqnarray*}
For the fourth integral, we compute 
\begin{eqnarray*}
	I_4 &=& \int_{0}^{\infty} \int_{0}^{u} \int_{c}^{\infty} \frac{n}{N} [S]^{\frac{1}{n}}_0 \varphi(a-c) \frac{f(a)}{\xi(a-c)} e^{-\tau u} da dc du \nonumber \\
	&=& \int_{0}^{\infty} \int_{c}^{\infty} \int_{c}^{\infty} \frac{n}{N} [S]^{\frac{1}{n}}_0 \varphi(a-c) \frac{f(a)}{\xi(a-c)} e^{-\tau u} da du dc \nonumber \\
	&=& \int_{0}^{\infty} \int_{c}^{\infty} \int_{c}^{\infty} \frac{n}{N} [S]^{\frac{1}{n}}_0 \varphi(a-c) \frac{f(a)}{\xi(a-c)} e^{-\tau u} du da dc \nonumber \\
	&=& \int_{0}^{\infty} \int_{c}^{\infty} \frac{n}{N} [S]^{\frac{1}{n}}_0 \varphi(a-c) \frac{f(a)}{\xi(a-c)} \left[ \frac{e^{-\tau u}}{-\tau}\right]_{c}^{\infty} da dc \nonumber \\
	&=& \frac{1}{\tau}\int_{0}^{\infty} \int_{c}^{\infty} \frac{n}{N} [S]^{\frac{1}{n}}_0 \varphi(a-c) \frac{f(a)}{\xi(a-c)} e^{-\tau c} da dc.
\end{eqnarray*}
Having a small amount of initial infecteds (i.e. $[I](0)=\int_{0}^{\infty}\varphi(a)da<<1$), the integrals $I_1$ and $I_4$ are approximately zero.
\subsection{Calculations for proof of equivalence in the case of gamma distribution} \label{appendixgammadistribution}
Here, we obtain the formula (\ref{eq:gammaeqIj}) by induction from variation of constants formulae (\ref{eq:variationofconstantsI1}) and (\ref{eq:variationofconstantsIj}). Letting $j=1$ in (\ref{eq:gammaeqIj}), we have
\begin{eqnarray*}
	[I_1](t)=\int_{0}^{t}\tau [SI](t-a)e^{-K \gamma a}da + [I]_0 e^{-K \gamma t},
\end{eqnarray*}
which comes directly from (\ref{eq:variationofconstantsI1}). Assuming that (\ref{eq:gammaeqIj}) holds for $1<j$, we prove that it holds for $j+1$. Indeed, we can do the following elaboration:
\begin{eqnarray*}
	[I_{j+1}](t) &=& \!\!\![I_{j+1}](0) e^{-K \gamma t}+\int_{0}^{t} e^{-K \gamma (t-s)}K \gamma [I_{j}](s)ds\nonumber\\
	&=&\!\!\!\int_{0}^{t}\!\! e^{-K \gamma (t-s)}\!K \gamma \!\!\left(\int_{0}^{s}\tau [SI](s-a)\frac{(K \gamma)^{j-1}}{(j-1)!}a^{j-1}e^{-K \gamma a}da\!+\![I]_0 \frac{(K \gamma)^{j-1}}{(j-1)!}s^{j-1}\!e^{-K \gamma s}\right)ds\nonumber\\
	&=&\!\!\!\int_{0}^{t}\!\!\frac{(K \gamma)^{j}}{(j-1)!} e^{-K \gamma (t-s)} \left( \int_{0}^{s}\tau [SI](s-a) a^{j-1}e^{-K \gamma a}da\;ds\right)\!\!+ [I]_0 \frac{(K \gamma)^{j}}{(j-1)!}e^{-K \gamma t}\int_{0}^{t}s^{j-1}ds\nonumber\\
	&=&\!\!\!\int_{0}^{t}\frac{(K \gamma)^{j}}{(j-1)!} e^{-K \gamma (t-s)} \left( \int_{0}^{s}\tau [SI](u) (s-u)^{j-1}e^{-K \gamma (s-u)}du\right) ds + [I]_0 \frac{(K \gamma)^{j}}{j!}t^je^{-K \gamma t}\nonumber\\
	&=&\!\!\!\int_{0}^{t}\frac{(K \gamma)^{j}}{(j-1)!}e^{-K \gamma (t-u)}\tau [SI](u)\left( \int_{u}^{t} (s-u)^{j-1}ds\right)du  + [I]_0 \frac{(K \gamma)^{j}}{j!}t^je^{-K \gamma t}\nonumber\\
	&=&\!\!\!\int_{0}^{t}\tau [SI](u)\frac{(K \gamma)^{j}}{j!}(t-u)^{j}e^{-K \gamma (t-u)}du  + [I]_0 \frac{(K \gamma)^{j}}{j!}t^j e^{-K \gamma t}\nonumber\\
	&=&\!\!\!\int_{0}^{t}\tau [SI](t-a)\frac{(K \gamma)^{j}}{j!}a^{j}e^{-K \gamma a}da  + [I]_0 \frac{(K \gamma)^{j}}{j!}t^je^{-K \gamma t}.\nonumber\\	
\end{eqnarray*}
\subsection{Implementation of numerical simulations}
As a validation of our models, we implemented an event-based stochastic algorithm for simulating the non-Markovian SIR process with arbitrary recovery time. In the code, the waiting times for possible events are generated from appropriate distributions. After selecting the smallest waiting time, the associated event (infection or recovery) happens, and according to the type of event, the necessary updates are executed. \\
For the numerical solution of integro-differential equations (\ref{eq:closeq}) and (\ref{eq:closeformmf}), we developed a numerical scheme based on collocation method. The numerical methods in \cite{brunner} were adapted to the reduced, but highly nonlinear pairwise system, having the following general form:
\begin{eqnarray}
\label{eq:main}
x'(t) &=& f(x(t),y(t)),\nonumber\\
y'(t) &=& g(x(t),y(t))-\int_0^t F(t-a,x(a),y(a)) e^{-\int_a^t G(x(s),y(s)) ds}da\\
&&-H\left(t,\int_0^t G(x(s),y(s))ds\right)\nonumber.
\end{eqnarray}
A collocation solution $u_h$ to a functional equation on an interval $I$ is an element from some finite-dimensional function space (the collocation space) which satisfies the equation on an appropriate finite subset of points in $I$ (the set of collocation points), whose cardinality essentially matches the dimension of the collocation space. For integro-differential equations the collocation equations are not yet in a form amenable to numerical computation, due to the presence of the memory term given by the integral operator, thus another discretisation step, based on appropriate quadrature approximations, is necessary to obtain the fully discretised collocation scheme. The presence of the double integral in (\ref{eq:main}) makes our scheme different from standard methods.

\end{document}